\documentclass{amsart}
\usepackage{amsmath}
\usepackage{amssymb}
\usepackage{amscd}
\usepackage{color}
\usepackage[all]{xy}
\newcommand{\A}{\mathcal{A}}
\newcommand{\Max}{\operatorname{Max}}
\newtheorem{theorem}{Theorem}
\newtheorem{lemma}[theorem]{Lemma}
\newtheorem{corollary}[theorem]{Corollary}
\newtheorem{proposition}[theorem]{Proposition}

\newtheorem{remark}[theorem]{Remark}

\begin{document}

\title[Graded integral UMT-domains]
{Graded integral domains which are UMT-domains}
\author[G.W. Chang and P. Sahandi]
{Gyu Whan Chang and Parviz Sahandi}

\address{(Chang) Department of Mathematics Education, Incheon National University,
Incheon 22012, Korea.} \email{whan@inu.ac.kr}

\address{(Sahandi) Department of Pure Mathematics, Faculty of Mathematical Sciences, University of Tabriz, Tabriz, Iran.} 
\email{sahandi@ipm.ir}

\date{\today}


\begin{abstract}
Let  $\Gamma$ be a torsionless commutative cancellative monoid,
$R =\bigoplus_{\alpha \in \Gamma}R_{\alpha}$ be a $\Gamma$-graded
integral domain, and $H$ be the set of nonzero homogeneous elements
of $R$. In this paper, we show that if $Q$ is a maximal $t$-ideal of
$R$ with $Q \cap H = \emptyset$, then $R_Q$ is a valuation domain.
We then use this result to give simple proofs of the facts 
that (i) $R$ is a UMT-domain if and only if $R_Q$
is a quasi-Pr\"{u}fer domain for each homogeneous maximal $t$-ideal $Q$ of $R$
and (ii) $R$ is a P$v$MD if and only if every nonzero finitely generated homogeneous ideal of $R$
is $t$-invertible, if and only if $R_Q$ is a valuation domain for all homogeneous maximal $t$-ideals $Q$ of $R$.
Let $D[\Gamma]$ be the monoid domain of $\Gamma$ over an integral domain $D$.
We also show that $D[\Gamma]$ is a UMT-domain if and only if $D$ is a UMT-domain
and the integral closure of $\Gamma_S$ is a valuation monoid for all maximal $t$-ideals $S$ of $\Gamma$.
Hence, $D[\Gamma]$ is a P$v$MD if and only if $D$ is a P$v$MD and $\Gamma$ is a P$v$MS.
\end{abstract}

\maketitle

\section{Introduction}

Let $\Gamma$ be a torsionless commutative cancellative monoid,
$R = \bigoplus_{\alpha\in\Gamma}R_{\alpha}$ be a $\Gamma$-graded
integral domain, and $H$ be the set of nonzero homogeneous elements of $R$. 
In \cite{cpre}, the first-named author generalized
the notion of UMT-domains to graded integral domains $R$ such that $R_H$ is a UFD as follows:
(i) A nonzero prime ideal $Q$ of $R$ is an upper to zero in $R$ if $Q = fR_H \cap R$ for some $f \in R$
and (ii) $R$ is a graded UMT-domain if every upper to zero in $R$ is a maximal $t$-ideal of $R$.
(Necessary definitions and notations will be reviewed in Sections 1.1 and 1.2.)
Among other things, he showed that if $R$ has a unit of nonzero degree, then 
$R$ is a graded UMT-domain if and only if $R$ is a UMT-domain.
In \cite{hs16}, the second-named author with Hamdi also studied graded integral domains which are also UMT-domains
(in a more general setting of semistar operations).
In particular, they proved that $R$ is a UMT-domain if and only if $R_Q$ is a quasi-Pr\"ufer domain
(i.e., a UMT-domain whose maximal ideals are $t$-ideals) for all homogeneous maximal $t$-ideals
$Q$ of $R$, if and only if $Q = (Q \cap R)[X]$ for all prime ideals $Q$ of the polynomial ring $R[X]$ with $Q \subseteq P[X]$
for some homogeneous prime $t$-ideal $P$ of $R$. 
This paper is a continuation of our study of graded integral domains which are UMT-domains.

More precisely, in
Section 2, we show that if $Q$ is a maximal $t$-ideal of $R$ with $Q \cap H = \emptyset$,
then $R_Q$ is a valuation domain.  
Then, among other things,
we use this result to give simple proofs of the facts 
that (i) $R$ is a UMT-domain if and only if $R_Q$
is a quasi-Pr\"{u}fer domain for each homogeneous maximal $t$-ideal $Q$ of $R$
and (ii) $R$ is a P$v$MD if and only if every nonzero finitely generated homogeneous ideal of $R$
is $t$-invertible, if and only if $R_Q$ is a valuation domain for all homogeneous maximal $t$-ideals $Q$ of $R$.
Let $D[\Gamma]$ be the monoid domain of $\Gamma$ over an integral domain $D$.
In Section 3, we prove that $D[\Gamma]$ is a UMT-domain if and only if $D$ is a UMT-domain
and the integral closure of $\Gamma_S$ is a valuation monoid for all maximal $t$-ideals $S$ of $\Gamma$.
Hence, we recover that  
$D[\Gamma]$ is a P$v$MD if and only if $D$ is a P$v$MD and $\Gamma$ is a P$v$MS  \cite[Proposition 6.5]{AA}.

\subsection{The $t$-operation and UMT-domains}

Let $D$ be an integral domain with quotient
field $K$ and $F(D)$ be the set of nonzero fractional ideals of $D$.
An {\em overring} of $D$ means a subring of $K$ containing $D$.
Let $T$ be an overring of $D$ and $P$ be a prime ideal of $D$.
Then $D \setminus P$ is a multiplicative set of both $D$ and $T$, 
and throughout this paper, we denote $T_P = T_{D \setminus P}$
as in the case of $D_P = D_{D \setminus P}$.

For $A \in F(D)$, let $A^{-1} = \{x \in K \mid xA \subseteq D\}$.
Then $A^{-1} \in F(D)$, and so we can define $A_v = (A^{-1})^{-1}$. 
The $t$-operation is defined by $A_t = \bigcup\{I_v \mid I \subseteq A$ is a nonzero finitely generated ideal of $D\}$ and 
the $w$-operation is defined by 
$A_w =\{x\in K \mid xJ\subseteq A$  for some nonzero finitely generated ideal $J$ of $D$ with $J_v =D\}.$
Clearly, $A \subseteq A_w \subseteq A_t \subseteq A_v$ for all $A \in F(D)$.
 An $A \in F(D)$ is called a {\em $t$-ideal} (resp., $w$-ideal)
if $A_t = A$ (resp., $A_w = A$). A $t$-ideal (resp., $w$-ideal) is a {\em maximal $t$-ideal}
(resp., {\em maximal $w$-ideal}) if it is maximal among proper integral $t$-ideals (resp., $w$-ideals). 
Let $t$-Max$(D)$ be the set of maximal $t$-ideals of $D$.
It is well known that $t$-Max$(D) \neq \emptyset$ when $D$ is not a field
and $t$-Max$(D) = w$-Max$(D)$, the set of maximal $w$-ideals of $D$.
Also, if $* = t, w$, then each maximal $*$-ideal is a prime ideal;
each proper $*$-ideal is contained in a maximal $*$-ideal; and each
prime ideal minimal over a $*$-ideal is a $*$-ideal.
An $I \in  F(D)$ is said to be {\em $t$-invertible} if $(II^{-1})_t = D$. We say that $D$
is a {\em Pr\"ufer $v$-multiplication domain} (P$v$MD) if each nonzero finitely generated
ideal of $D$ is $t$-invertible. 

Let $X$ be an indeterminate over $D$ and $D[X]$ 
be the polynomial ring over $D$. A nonzero prime ideal of $D[X]$
is said to be an {\em upper to zero in $D[X]$} if $Q \cap D = (0)$.
Clearly, $Q$ is an upper to zero in $D[X]$ if and only if $Q = fK[X] \cap D[X]$ for some $f \in D[X]$.
As in \cite{hz89}, we say that $D$ is a UMT-domain if every upper to zero in $D[X]$ is
a maximal $t$-ideal of $D[X]$. It is known that $D$ is a P$v$MD if and only if $D$  is
an integrally closed UMT-domain \cite[Proposition 3.2]{hz89}; hence
UMT-domains can be considered as non-integrally closed P$v$MDs. 
An integral domain is said to be {\em quasi-Pr\"ufer} if its integral
closure is a Pr\"ufer domain \cite[Corollary 6.5.14]{fhp97}.
It is known that $D$ is quasi-Pr\"ufer if and only if $D$ is a UMT-domain whose
maximal ideals are $t$-ideals \cite[Lemma 2.1 and Theorem 2.4]{dhlrz}. 

Let $M$ be a commutative cancellative monoid. Then,
as in the integral domain case, we can define the $t$-operation on $M$, the $t$-invertiblilty
of ideals of $M$, and Pr\"ufer $v$-multiplication semigroup (P$v$MS).
The reader can refer to \cite[Sections 32 and 34]{gilmer} 
and \cite[Part B]{hk} for more on the $t$-operations on integral domains and monoids,
respectively.

\subsection{Graded integral domains}
Let $\Gamma$ be a (nonzero) torsionless commutative cancellative monoid (written
additively) and $\langle \Gamma \rangle = \{a - b \mid a,b \in \Gamma\}$ be the
quotient group of $\Gamma$; so $\langle \Gamma \rangle$ is a torsionfree abelian group.
It is well known that a cancellative monoid $\Gamma$ is torsionless
if and only if $\Gamma$ can be given a total order compatible with the monoid
operation \cite[page 123]{no68}. By a $(\Gamma$-)graded integral domain  $R =\bigoplus_{\alpha \in \Gamma}R_{\alpha}$,
we mean an integral domain graded by $\Gamma$.
That is, each nonzero $x \in R_{\alpha}$ has degree $\alpha$, i.e., deg$(x) = \alpha$,  and
deg$(0) = 0$. Thus, each nonzero $f \in R$ can be written uniquely as $f = x_{\alpha_1} + \dots + x_{\alpha_n}$ with
deg$(x_{\alpha_i}) = \alpha_i$ and $\alpha_1 < \cdots < \alpha_n$.
A nonzero $x \in R_{\alpha}$ for every $\alpha \in \Gamma$ is said
to be {\em homogeneous}. 

Let  $H = \bigcup_{\alpha \in \Gamma}(R_{\alpha} \setminus \{0\})$; so
$H$ is the saturated multiplicative set of nonzero homogeneous elements of $R$.
Then $R_H$, called the {\em homogeneous quotient field} of $R$,
 is a $\langle \Gamma \rangle$-graded integral domain whose nonzero homogeneous elements are units.
We say that an overring $T$ of $R$ is a {\em homogeneous overring} of $R$ if
$T = \bigoplus_{\alpha \in \langle \Gamma \rangle}(T \cap (R_H)_{\alpha})$;
so $T$ is a $\langle \Gamma \rangle$-graded integral domain
such that $R \subseteq T \subseteq R_H$. 
Let $\bar{R}$ be the integral closure of $R$. Then $\bar{R}$ is a homogeneous overring of $R$
(cf. \cite[Theorem 2.10]{j83} or \cite[Lemma 1.6]{cpre}).
Also, $R_S$ is a homogeneous overring of $R$ for
a multiplicative set $S$ of nonzero homogeneous elements of $R$
(with deg$(\frac{a}{b}) =$ deg$(a) -$ deg$(b)$ for $a \in H$ and $b \in S$).

For a fractional ideal $A$ of $R$ with $A \subseteq R_H$,
let $A^*$ be the fractional ideal of $R$ generated by homogeneous elements in $A$; so
$A^* \subseteq A$. The fractional ideal $A$ is said to be {\em homogeneous} if $A^* = A$.
A homogeneous ideal (resp., homogeneous $t$-ideal) of $R$ is called a {\em homogeneous maximal ideal}
(resp., {\em homogeneous maximal $t$-ideal}) if it is
maximal among proper homogeneous ideals (resp., homogeneous $t$-ideals) of $R$.
It is easy to see that a homogeneous maximal ideal need not be a maximal ideal, while
a homogeneous maximal $t$-ideal is a maximal $t$-ideal \cite[Lemma 1.2]{ac05}.
Also, it is easy to see that each proper homogeneous ideal (resp., homogeneous $t$-ideal) of $R$
is contained in a homogeneous maximal ideal (resp., homogeneous maximal $t$-ideal) of $R$.
A maximal $t$-ideal $Q$ of $R$ is homogeneous if and only if $Q \cap H \neq \emptyset$
\cite[Lemma 1.2]{ac05}; hence $\Omega = \{Q \in t$-Max$(D) \mid Q \cap H \neq \emptyset\}$
is the set of homogeneous maximal $t$-ideals of $R$.

For $f \in  R_H$, let $C(f)$ denote the fractional
ideal of $R$ generated by the homogeneous components of $f$. For a fractional ideal $A$ of $R$
with $A \subseteq R_H$, let
$C(A) = \sum_{f \in A} C(f )$. It is clear that both $C(f)$ and $C(A)$ are homogeneous fractional ideals of $R$. 
Let $N(H) = \{f \in R\mid C(f)_v = R\}$.
It is well known that if $f, g \in R_H$, then $C(f)^{n+1}C(g) = C(f)^nC(fg)$ for some integer $n \geq 1$
\cite{no68}; so $N(H)$ is a saturated multiplicative subset of $R$.
As in \cite{ac13}, we say that $R$ satisfies property $(\#)$ if $C(I)_t = R$ implies $I \cap N(H) \neq \emptyset$
for all nonzero ideals $I$ of $R$; equivalently,
Max$(R_{N(H)}) = \{Q_{N(H)} \mid Q \in \Omega\}$ \cite[Proposition 1.4]{ac13}.
It is known that $R$ satisfies property $(\#)$ if $R$ is one of the following integral domains:
(i) $R$ contains a unit of nonzero degree and (ii) $R = D[\Gamma]$ is the monoid domain
of $\Gamma$ over an integral domain $D$ \cite[Example 1.6]{ac13}.
We say that  $R$ is a {\em graded-Pr\"ufer domain} 
if each nonzero finitely generated homogeneous ideal of  $R$ is invertible
and that $R$ is a {\em graded-valuation domain}
if either $\frac{a}{b} \in R$ or $\frac{b}{a} \in R$ for all $a,b \in H$. Clearly, 
a graded-valuation domain is a graded-Pr\"ufer domain.
The graded integral domain $R$ is a {\em graded Krull domain}
if and only if every nonzero homogeneous (prime) ideal of $R$ is $t$-invertible \cite[Theorem 2.4]{ac05}.
Then $R$ is a Krull domain if and only if $R$ is a graded Krull domain and $R_H$ is a Krull domain \cite[Theorem 5.8]{aa82}.
Also, a graded Krull domain is a P$v$MD \cite[Theorem 6.4]{AA}.


\section{Graded integral domains and UMT-domains}

Let $\Gamma$ be a (nonzero) torsionless commutative cancellative monoid,
$R=\bigoplus_{\alpha\in\Gamma}R_{\alpha}$ be an
integral domain graded by $\Gamma$, $\bar{R}$ be the integral closure of $R$,
$H$ be the set of nonzero homogeneous elements of $R$,
and $N(H) = \{f \in R \mid C(f)_v = R\}$.
Since $R$ is an integral domain,
we may assume that $R_{\alpha} \neq \{0\}$ for all $\alpha \in \Gamma$.

\begin{lemma} \label{lemma 1} $($\cite[Proposition 2.1]{aa82}$)$
Let $R =\bigoplus_{\alpha \in \Gamma}R_{\alpha}$ be a graded integral domain and
$H$ be the set of nonzero homogeneous elements of $R$. Then $R_H$ is a completely
integrally closed GCD-domain.
\end{lemma}

We next give the main result of this section. For the proof of this result,
we recall the following well-known theorem:
{\em Let $S$ be a multiplicative set of an integral domain $D$.
If $I$ is a nonzero ideal of $D$ such that $ID_S$ is a
$t$-ideal of $D_S$, then $ID_S \cap D$ is a $t$-ideal of $D$ \cite[Lemma 3.17]{k89}.}
Thus, if $Q$ is a maximal $t$-ideal of $D$, then $QD_S$ is a maximal $t$-ideal of $D_S$
if and only if $(QD_S)_t \subsetneq D_S$.

\begin{theorem} \label{theorem 1}
Let $R =\bigoplus_{\alpha \in \Gamma}R_{\alpha}$ be a graded integral domain,
$H$ be the set of nonzero homogeneous elements of $R$, and
$Q$ be a prime $t$-ideal of $R$ with $C(Q)_t = R$ $($e.g., $Q$ is a maximal $t$-ideal of $R$
such that $Q \cap H = \emptyset)$.
\begin{enumerate}
\item $R_Q$ is a valuation domain.
\item $QR_H$ is a prime $t$-ideal of $R_H$. Hence,
if $Q$ is a maximal $t$-ideal of $R$, then $QR_H$ is a maximal $t$-ideal of $R_H$.
\end{enumerate}
\end{theorem}

\begin{proof}
(1) Let $T = R[X,X^{-1}]$ be the Laurent polynomial ring over $R$ and
$N(T) = \{aX^n \mid a \in H$ and $n \in \mathbb{Z}\}$. Clearly, $T$
is a $(\Gamma \oplus \mathbb{Z})$-graded integral domain with
deg$(aX^n) = (\alpha, n)$ for $0 \neq a \in R_{\alpha}$ and $n \in
\mathbb{Z}$, and $N(T)$ is the saturated multiplicative set of nonzero
homogeneous elements of $T$. Also, $QT$ is a prime $t$-ideal of
$T$  \cite[Proposition 4.3]{hh} and $QT \cap N(T) =
\emptyset$; so $QT_{N(T)} \subsetneq T_{N(T)}$. Assume that 
$(QT_{N(T)})_t = T_{N(T)}$.
Then there is a nonzero finitely generated subideal $J$ of $Q$
such that $T_{N(T)} = (JT_{N(T)})_v = (JT_{N(T)})^{-1} =
J^{-1}T_{N(T)}$.

Let $N = \{f \in T \mid C(f)_v = T\}$. 
Since $C(Q)_t = R$, there is a nonzero finitely generated subideal $J'$ of $Q$ such that $C(J')_t = R$.
Note that $T$ has a unit of nonzero degree and $C(J'T)_t = C(J')_tT = T$;
hence $J'T \cap N \neq \emptyset$. So if 
$I = J + J'$, then $I$ is a finitely generated subideal
of $Q$ such that $I^{-1}T_{N(T)} = (IT_{N(T)})^{-1} = T_{N(T)}$ and
$I^{-1}T_N = (IT_N)^{-1} = (T_N )^{-1} = T_N$. Thus, $(IT)^{-1} = I^{-1}T
\subseteq I^{-1}T_{N(T)} \cap I^{-1}T_N = T_{N(T)} \cap T_N = T$
(cf. \cite[Lemma 1.2(2)]{ac13} for the last equality),
and so $(IT)^{-1} = T$. Thus, $QT = (QT)_t \supseteq (IT)_v = T$, a
contradiction. Therefore, $(QT_{N(T)})_t \subsetneq T_{N(T)}$,
and since $T_{N(T)}$ is a GCD-domain (so a P$v$MD) by
Lemma \ref{lemma 1}, $(QT_{N(T)})_t = QT_{N(T)}$, and hence $T_{QT} = (T_{N(T)})_{QT_{N(T)}}$ is a
valuation domain. Note that
$$T_{QT} = R[X,X^{-1}]_{QR[X,X^{-1}]} = R_Q[X]_{QR_Q[X]} = R_Q(X);$$ so
$T_{QT} \cap qf(R) = R_Q$, where $qf(R)$ is the quotient field of $R$.
Thus, $R_Q$ is a valuation domain.

(2)  By (1), $R_Q$ is a valuation domain, and hence $QR_Q$ is a $t$-ideal of $R_Q$.
Note that $QR_H = QR_Q \cap R_H$. Thus, $QR_H$ is a prime $t$-ideal of $R_H$. 
Also, if $Q$ is a maximal $t$-ideal of $R$, then
$QR_H$ is a maximal $t$-ideal of $R_H$. 
\end{proof}

\begin{corollary} \label{coro 3}
Let $R =\bigoplus_{\alpha \in \Gamma}R_{\alpha}$ be a graded
integral domain such that $R_H$ is a UFD. If $Q$ is a maximal
$t$-ideal of $R$ such that $Q\cap H = \emptyset$, then $R_Q$ is a
rank-one DVR.
\end{corollary}

\begin{proof}
By Theorem \ref{theorem 1}, $QR_H$ is a prime $t$-ideal of $R_H$,
and since $R_H$ is a UFD, $R_Q = (R_H)_{QR_H}$ is a rank-one DVR.
\end{proof}

\begin{corollary}
{\em (cf. \cite[Lemma 2.3]{fgh98})}
Let $D$ be an integral domain, $\{X_{\alpha}\}$ be a nonempty set of indeterminates over $D$,
and $Q$ be a maximal $t$-ideal
of $D[\{X_{\alpha}\}]$ with $Q \cap D = (0)$. Then
$D[\{X_{\alpha}\}]_Q$ is a rank-one DVR.
\end{corollary}

\begin{proof}
Clearly, if $R = D[\{X_{\alpha}\}]$, then $R$ is a graded integral
domain such that $R_H = K[\{X_{\alpha}, X_{\alpha}^{-1}\}]$ is a
UFD, where $K$ is the quotient field of $D$. Also, $Q \cap D = (0)$ implies that either $Q \cap H = \emptyset$
or $X_{\alpha} \in Q$ for some $\alpha$. If $X_{\alpha} \in Q$, then
$Q = X_{\alpha}R$ because $X_{\alpha}$ is a prime element, and thus
$R_Q$ is a rank-one DVR. Next, if $Q \cap H = \emptyset$, then
$R_Q$ is a rank-one DVR by Corollary \ref{coro 3}.
\end{proof}

Let $T$ be an overring of an integral domain $D$.
As in \cite{dhlz89}, we say that $T$ is {\em $t$-linked over} $D$ if $(IT)_v=T$
for each nonzero finitely generated ideal $I$ of $D$ with $I_v = D$.
Clearly, $D_S$ is $t$-linked over $D$ for every multiplicative set $S$ of $D$.

\begin{lemma} \label{t-linked}
Let $T$ be a homogeneous overring of $R =\bigoplus_{\alpha \in \Gamma}R_{\alpha}$.
Then $T$ is $t$-linked over $R$ if and only if $I_v = R$ implies $(IT)_v = T$
for all nonzero finitely generated homogeneous ideals $I$ of $R$.
\end{lemma}

\begin{proof}
$(\Rightarrow)$ Clear. 

$(\Leftarrow)$ Let $A$ be a nonzero finitely generated ideal of $R$
with $A_v = R$. It suffices to show that $AT \nsubseteq Q$ for all maximal $t$-ideals $Q$ of $T$.
Let $H$ be the set of nonzero homogenous elements of $R$ and 
let $H' = \bigcup_{\alpha \in \langle \Gamma \rangle}((R_H)_{\alpha} \cap T \setminus \{0\})$
be the set of nonzero homogeneous elements of $T$.
If $Q \cap H' = \emptyset$, then $Q_{H'}$ is a maximal $t$-ideal of $T_{H'}$ by Theorem \ref{theorem 1}. 
Note that $Q_{H'} = (Q \cap R)R_H$ because $R_H = T_{H'}$; so $(Q \cap R)R_H$ is a $t$-ideal of $R_H$. 
Since $(Q \cap R)R_H \cap R = Q \cap R$, we have that $Q \cap R$ is a $t$-ideal of $R$.
Thus, $AT \nsubseteq Q$. Next, assume that $Q \cap H' \neq \emptyset$; so $Q$ is homogeneous \cite[Lemma 1.2]{ac05}.
Clearly, $C(A)$ is finitely generated and $C(A)_v = R$. Hence, $(C(A)T)_v = T$ by assumption,
and thus $C(A)T \nsubseteq Q$. Again, since $Q$ is homogeneous, $AT \nsubseteq Q$.
\end{proof}

Let $D$ be an integral domain with quotient field $K$.
It is well known that $x \in K$ is integral over $D$ if and only if
$xI \subseteq I$ for some nonzero finitely generated ideal $I$ of $D$.
As the $w$-operation analog, we say that $x \in K$ is {\em $w$-integral over $D$}
if and only if $xI_w \subseteq I_w$ for some nonzero finitely generated ideal $I$ of $D$.
Let $D^{[w]} = \{x \in K \mid x$ is $w$-integral over $D\}$. Then
$D^{[w]}$, called the {\em $w$-integral closure} of $D$, is an integrally closed $t$-linked overring of $D$
\cite[Lemma 1.2 and Theorem 1.3(1)]{cz06}.
It is known that if $D'$ is an integrally closed $t$-linked overring of $D$, then $D^{[w]} \subseteq D'$ 
\cite[Proposition 2.13(b)]{dhlz89}.

\begin{theorem} \label{w-integral}
Let $R =\bigoplus_{\alpha \in \Gamma}R_{\alpha}$ be a graded integral domain,
$\bar{R}$ be the integral closure of $R$,
and $\Omega = \{Q \in t$-Max$(R) \mid Q \cap H \neq \emptyset\}$.
\begin{enumerate}
\item $R^{[w]} = R_H \cap (\bigcap_{Q \in \Omega}\bar{R}_Q) = \bigcap_{Q \in \Omega}\bar{R}_{H \setminus Q}$.
\item $(R^{[w]})_{H \setminus Q} = \bar{R}_{H \setminus Q}$ for all $Q \in \Omega$.
\item $R^{[w]}$ is a homogeneous $t$-linked overring of $R$.
\item If $R$ satisfies property $(\#)$, then $R^{[w]} = R_H \cap \bar{R}_{N(H)}$.
\end{enumerate}
\end{theorem}

\begin{proof}
(1) If $Q \in \Omega$, then $\bar{R}_{H \setminus Q} \subseteq \bar{R}_Q$ and
$\bar{R}_{H \setminus Q}$ is $t$-linked over $R$. (For if $I$ is a nonzero
finitely generated homogeneous ideal of $R$ with $I_v = R$, then $I \nsubseteq Q$,
and hence $I \cap (H \setminus Q) \neq \emptyset$; so $IR_{H \setminus Q} = R_{H \setminus Q}$.
Hence, $(I\bar{R}_{H \setminus Q})_v = I\bar{R}_{H \setminus Q} = \bar{R}_{H \setminus Q}$.)
Hence, $\bigcap \{\bar{R}_{H \setminus Q} \mid Q \in \Omega\}$ is $t$-linked over $R$ \cite[Proposition 2.2]{dhlz89}, and thus
\begin{eqnarray*}
R^{[w]} &=& (\bigcap\{\bar{R}_Q \mid Q \in t\text{-Max}(R) \text{ with } Q \cap H = \emptyset\})
\cap (\bigcap \{\bar{R}_Q \mid Q \in \Omega\})\\
&\supseteq& R_H \cap (\bigcap \{\bar{R}_Q \mid Q \in \Omega\})\\
&\supseteq& \bigcap \{\bar{R}_{H \setminus Q} \mid Q \in \Omega\}\\
&\supseteq& R^{[w]},
\end{eqnarray*}
where the first equality is from \cite[Theorem 1.3]{cz06}; 
the second containment follows
from the fact that if $Q$ is a maximal $t$-ideal of $R$ with $Q \cap H = \emptyset$, then $H \subseteq R \setminus Q$, and 
hence $R_H \subseteq R_Q \subseteq \bar{R}_Q$; 
the third containment follows from the fact that $\bar{R}$
is a homogeneous overring of $R$;
and the fourth containment is from \cite[Corollary 1.4]{cz06} because
$\bigcap \{\bar{R}_{H \setminus Q} \mid Q \in \Omega\}$ is an integrally closed $t$-linked overring of $R$.
Thus, the equalities hold.

(2) Clearly, $\bar{R} \subseteq R^{[w]}$. Thus, the result is an
immediate consequence of (1).

(3) Let $Q \in \Omega$. Then $\bar{R}_{H \setminus Q}$  is a homogeneous
$t$-linked overring of $R$ (see the proof of (1)).
Thus, $R^{[w]}$ is a homogeneous $t$-linked overring of $R$ by (1)
and \cite[Proposition 2.2]{dhlz89}.

(4) Clearly, $\bar{R}_{N(H)} \subseteq \bar{R}_Q$ for all $Q \in \Omega$,
and hence $R_H \cap \bar{R}_{N(H)} \subseteq R^{[w]}$ by (1). For the reverse containment,
let $A$ be a nonzero finitely generated ideal of $R$ with $A_v = R$. Then $A \nsubseteq Q$ for all
$Q \in \Omega$, and thus $AR_{N(H)} = R_{N(H)}$ because $R$ satisfies property $(\#)$;
so $A\bar{R}_{N(H)} = \bar{R}_{N(H)}$. Hence, $\bar{R}_{N(H)}$ is an integrally closed $t$-linked overring of $R$.
Thus, $R^{[w]} \subseteq R_H \cap (\bar{R}_{N(H)})^{[w]} = R_H \cap \bar{R}_{N(H)}$ \cite[Corollary 1.4]{cz06}.
\end{proof}

It is known that an integral domain $D$ is a UMT-domain if and only
if $D_P$ is quasi-Pr\"ufer for all maximal $t$-ideals $P$ of $D$, if
and only if the integral closure of $D_P$ is a Pr\"ufer domain for
all maximal $t$-ideals $P$ of $D$ \cite[Theorem 1.5]{fgh98}, if and only if the integral
closure of $D_P$ is a B\'ezout domain for all maximal $t$-ideals $P$
of $D$ \cite[Lemma 2.2]{c08}. Also, if $S$ is a multiplicative set of a UMT-domain $D$, then
$D_S$ is a UMT-domain \cite[Proposition 1.2]{fgh98}. 
The implication of (2) $\Rightarrow$ (1) of the next corollary appears in \cite{hs16}.

\begin{corollary} \label{coro 1}
The following statements are equivalent for $R =\bigoplus_{\alpha
\in \Gamma}R_{\alpha}$.
\begin{enumerate}
\item $R$ is a UMT-domain.
\item $R_Q$ is quasi-Pr\"ufer for all homogeneous maximal $t$-ideals $Q$ of $R$.
\item $\overline{R}_Q$ is a Pr\"ufer domain for all homogeneous maximal $t$-ideals $Q$ of $R$.
\item $\overline{R}_Q$ is a B\'ezout domain for all homogeneous maximal $t$-ideals $Q$ of $R$.
\item $R_{H \setminus Q}$ is a UMT-domain and $Q_{H \setminus Q}$ is a $t$-ideal for all homogeneous
maximal $t$-ideals $Q$ of $R$.
\item $\overline{R}_{H \setminus Q}$ is a graded-Pr\"ufer domain
for all homogeneous maximal $t$-ideals $Q$ of $R$.
\item $R^{[w]}$ is a P$v$MD, and
$(Q \cap R)_t \subsetneq R$ implies $Q_t \subsetneq R^{[w]}$ for all nonzero prime
ideals $Q$ of $R^{[w]}$ with $Q \cap R$ homogeneous.
\item $R_{N(H)}$ is a UMT-domain.  
\end{enumerate}
\end{corollary}

\begin{proof}
$(1) \Rightarrow (2) \Leftrightarrow$ (3) \cite[Theorem 1.5]{fgh98}.

$(2) \Rightarrow (1)$ Let $Q$ be a maximal $t$-ideal of $R$. If $Q
\cap H \neq \emptyset$, then $Q$ is homogeneous, and hence $R_Q$ is
quasi-Pr\"ufer by assumption. Next, if $Q \cap H = \emptyset$, then
$R_Q$ is a valuation domain by Theorem \ref{theorem 1}. Thus, $R$ is
a UMT-domain.

(2) $\Leftrightarrow$ (4) \cite[Lemma 2.2]{c08}.

(1) $\Rightarrow$ (5) \cite[Propositions 1.2 and 1.4]{fgh98}.

(5) $\Leftrightarrow$ (6) $\Leftrightarrow$ (2) \cite[Corollary 2.12]{hs16}.

(1) $\Rightarrow$ (7) \cite[Theorem 2.6]{cz06}.

(7) $\Rightarrow$ (1) Let $Q$ be a nonzero prime ideal of $R^{[w]}$
such that $(Q \cap R)_t \subsetneq R$. 
If $C(Q \cap R)_t \subsetneq R$, then there is a homogeneous maximal $t$-ideal  $P$ of $R$ such that
$Q \cap R \subseteq C(Q \cap R)_t \subseteq P$. Then there is a prime ideal $M$ of 
$R^{[w]}$ such that $Q \subseteq M$ and $M \cap R = P$ \cite[Corollary 1.4(3)]{cz06};
so by assumption, $M_t \subsetneq R^{[w]}$. Since $R^{[w]}$ is a P$v$MD,
$(R^{[w]})_M$ is a valuation domain, and thus $(R^{[w]})_Q$ is a valuation domain.
Thus, $Q$ is a $t$-ideal of $R^{[w]}$. 
Next, let $C(Q \cap R)_t = R$. There is a maximal $t$-ideal $P$ of $R$ such that $Q \cap R \subseteq P$.
Note that  $(R^{[w]})_Q$ is an overring of $R_P$ and $C(P)_t = R$.
Since $R_P$ is a valuation domain  by Theorem \ref{theorem 1}, 
$(R^{[w]})_Q$ is a valuation domain, and thus $Q$ is a $t$-ideal. 
Thus, $R$ is a UMT-domain \cite[Theorem 2.6]{cz06}.

(1) $\Rightarrow$ (8) Clear.

(8) $\Rightarrow$ (2) Let $Q$ be a homogeneous maximal $t$-ideal of $R$. Then
$Q_{N(H)}$ is a $t$-ideal \cite[Proposition 1.3]{ac13}, and thus $R_Q = (R_{N(H)})_{Q_{N(H)}}$
is a quasi-Pr\"ufer domain.
\end{proof}

\begin{corollary} \label{sharp}
Let $R =\bigoplus_{\alpha \in \Gamma}R_{\alpha}$ be a graded
integral domain with property $(\#)$. Then the following statements are equivalent.
\begin{enumerate}
\item $R$ is a UMT-domain.
\item $R_{N(H)}$ is a quasi-Pr\"ufer domain.
\item $\overline{R}_{N(H)}$ is a Pr\"ufer domain.
\end{enumerate}
\end{corollary}

\begin{proof}
This is an immediate consequence of Corollary \ref{coro 1} because
Max$(R_{N(H)}) = \{Q_{N(H)} \mid Q$ is a homogeneous maximal
$t$-ideal of $R\}$ and each maximal ideal of $R_{N(H)}$ is a
$t$-ideal \cite[Proposition 1.3(4)]{ac13}.
\end{proof}

Let $D$ be an integral domain.
It is known that $D$ is a P$v$MD if and only if
$D_P$ is a valuation domain for all maximal $t$-ideals $P$ of $D$ \cite[Theorem 3.2]{k89}.
Also,  a graded integral domain $R$ is a P$v$MD
if and only if every nonzero finitely generated homogeneous ideal of $R$
is $t$-invertible \cite[Theorem 6.4]{AA}, if and only if $R_Q$ is a valuation domain
for all homogeneous maximal $t$-ideals $Q$ of $R$ \cite[Lemma 2.7]{ckl10},
if and only if every nonzero ideal of $R$
generated by two homogeneous elements is $t$-invertible 
(note that if $A,B,C$ are ideals of $D$, then $(A+B)(B+C)(C+A) = (A+B+C)(AB+BC+CA)$). 
We next use Theorem \ref{theorem 1} to give simple proofs of these results.

\begin{corollary} \label{pvmd}
The following statements are equivalent for $R =\bigoplus_{\alpha \in \Gamma}R_{\alpha}$.
\begin{enumerate}
\item $R$ is a P$v$MD.
\item $R_Q$ is a valuation domain for all homogeneous maximal $t$-ideals $Q$ of $R$.
\item $R$ is integrally closed and $I_t = I_w$ for all nonzero homogeneous ideals $I$ of $R$.
\item Every homogeneous $t$-linked overring of $R$ is integrally closed.
\item $R_{H \setminus Q}$ is a graded-valuation domain for all homogeneous maximal $t$-ideals $Q$ of $R$.
\item $R_{N(H)}$ is a P$v$MD.
\item Every nonzero finitely generated homogeneous ideal of $R$ is $t$-invertible.
\end{enumerate}
\end{corollary}

\begin{proof}
(1) $\Rightarrow$ (2) Clear.

(2) $\Rightarrow$ (1) 
Note that if $Q$ is a maximal $t$-ideal of $R$, then either $Q \cap H = \emptyset$ or $Q$ is homogeneous; so
it suffices to show that if $Q$ is a maximal $t$-ideal of $R$ with $Q \cap H = \emptyset$,
then $R_Q$ is a valuation domain by assumption. Thus, the result
follows directly from Theorem \ref{theorem 1}. 

(1) $\Rightarrow$ (3) \cite[Theorem 3.5]{k89}.

(3) $\Rightarrow$ (1) Let $T = R[X,X^{-1}]$ be as in the proof of Theorem \ref{theorem 1}.
For $a,b \in H$, let $f = a+bX$ and $g = a-bX$. Then $C(fg) = (a^2, b^2)T$ and $C(f)C(g) = ((a,b)T)^2$
because $X$ is a unit of $T$. Since $R$ is integrally closed,
$(a^2,b^2)_wT = (a^2,b^2)_vT = ((a^2,b^2)T)_v = C(fg)_v = (C(f)C(g))_v = (((a,b)T)^2)_v =((a,b)^2)_vT = ((a,b)^2)_wT$
by assumption (see \cite[Theorem 3.5(2)]{AA} for the fourth equality),
and hence $(a^2,b^2)_w = ((a,b)^2)_w$. Thus, $(a,b)$ is $t$-invertible \cite[Lemma 4]{c13}.
Hence, $R$ is a P$v$MD.

(1) $\Rightarrow$ (4) \cite[Theorem 2.10]{dhlz89}.

(4) $\Rightarrow$ (5) Let $Q$ be a homogeneous maximal $t$-ideal of $R$. Then
$R_{H \setminus Q}$ is a homogeneous $t$-linked overring of $R$.
Note that if $T$ is a homogeneous overring of $R_{H \setminus Q}$, then $T$
is $t$-linked over $R$ by Lemma \ref{t-linked} and the proof of Theorem \ref{w-integral}(1); hence $T$ is integrally closed
by assumption. Hence, $R_{H \setminus Q}$ is a graded-Pr\"ufer domain 
\cite[Corollary 3.8]{acz16}, and thus
$R_{H \setminus Q}$ is a graded-valuation domain \cite[Lemma 4.3]{s14}.

(5) $\Rightarrow$ (1) \cite[Corollary 4.9]{s14}.

(1) $\Rightarrow$ (6) Clear.

(6) $\Rightarrow$ (2) Let $Q$ be a homogeneous maximal $t$-ideal of $R$. Then
$Q_{N(H)}$ is a maximal $t$-ideal of $R_{N(H)}$ \cite[Proposition 1.3]{ac13}, and 
thus $R_Q = (R_{N(H)})_{Q_{N(H)}}$ is a valuation domain.

(1) $\Rightarrow$ (7) Clear.

(7) $\Rightarrow$ (1) Let $I$ be a nonzero finitely generated ideal of  $R$.
It suffices to show that $I$ is $t$-locally principal \cite[Proposition 2.6]{k89}. Let $Q$ be a maximal
$t$-ideal of $R$. If $Q \cap H = \emptyset$, then $R_Q$ is a valuation domain by Theorem \ref{theorem 1},
and since $I$ is finitely generated, $IR_Q$ is principal. Next, assume that $Q \cap H \neq \emptyset$; so
$Q$ is homogeneous \cite[Lemma 1.2]{ac05}.
Let $A = \sum_{f \in I}C(f)$. 
Then, since $I$ is finitely generated,
$A$ is a finitely generated homogeneous ideal of $R$, and so $A$ is $t$-invertible by assumption. 
Hence, $AR_Q = x_{\alpha}R_Q$ for a homogeneous component $x_{\alpha}$ of some $f \in I$ \cite[Proposition 7.4]{gilmer}.
Note that $f \in I \subseteq IR_Q \subseteq AR_Q = x_{\alpha}R_Q$, and thus
$\frac{f}{x_{\alpha}} \in R_Q$. Assume that  $\frac{f}{x_{\alpha}} \in QR_Q$.
Then $gf \in x_{\alpha}Q$ for some $g \in R \setminus Q$. By \cite[Lemma 1.1(1)]{ac13}, $C(g)^{n+1}C(f) = C(g)^nC(gf)$
for some integer $n \geq 1$, and hence $x_{\alpha}C(g)^{n+1} \subseteq C(gf) \subseteq x_{\alpha}Q$ 
because $x_{\alpha}Q$ is a homogeneous ideal of $R$. Therefore, $C(g) \subseteq Q$, a contradiction.
Thus, $\frac{f}{x_{\alpha}} \in R_Q \setminus QR_Q$, and hence $x_{\alpha} \in IR_Q$.
Thus, $IR_Q = x_{\alpha}R_Q$.
\end{proof}

Let $R[X]$ be the polynomial ring over $R=\bigoplus_{\alpha\in\Gamma}R_{\alpha}$.
For a polynomial $f= f_0+f_1X+\cdots+f_nX^n\in R[X]$, define the
\emph{homogeneous content ideal of $f$} by
$\A_f =\sum_{i=0}^nC(f_i)$. Let $N(v) =\{f\in R[X]\mid
f\neq0\text{ and }(\A_f)_v=R\}$. It is known that
$N(v)=R[X]\backslash\bigcup\{Q[X]\mid Q$ is a homogeneous maximal
$t$-ideal of $R\}$ is multiplicatively closed and
$\Max(R[X]_{N(v)})=\{Q[X]_{N(v)}\mid Q$ is a homogeneous maximal
$t$-ideal of $R\}$ \cite[Proposition 2.3]{s16}.

\begin{lemma}\label{u}$($\cite[Proposition 3.4]{hs16}$)$ Let $R=\bigoplus_{\alpha\in\Gamma}R_{\alpha}$ be a
graded integral domain. Then 
$R$ is a UMT-domain if and only if
every prime ideal of $R[X]_{N(v)}$ is extended from a homogeneous prime ideal of $R$. 
\end{lemma}

The next result is a graded integral domain analog of \cite[Corollary 1.6]{fgh98} that
if $P \subseteq Q$ are nonzero prime ideals of a UMT-domain such that $Q$ is a $t$-ideal,
then $P$ is a $t$-ideal.  

\begin{proposition}\label{homo} 
Let $R =\bigoplus_{\alpha \in
\Gamma}R_{\alpha}$ be a UMT-domain. If $Q$ is a nonzero prime ideal
of $R$ such that $C(Q)_t \subsetneq R$, then $Q$ is a homogeneous prime $t$-ideal.
\end{proposition}

\begin{proof}
Let $M$ be a maximal $t$-ideal of $R$ such that $C(Q)_t\subseteq M$.
Then $M$ is homogeneous \cite[Lemma 1.2]{ac05} and $Q$ is a $t$-ideal \cite[Corollary 1.6]{fgh98}
because $R$ is a UMT-domain and $Q \subseteq M$. Also,
$Q[X]\subseteq M[X]$, and hence $Q[X]\cap N(v)=\emptyset$.
Hence, $Q[X]_{N(v)}$ is a proper prime ideal of $R[X]_{N(v)}$. Thus,
since $R$ is a UMT-domain, there is a homogeneous prime ideal $P$ of
$R$ such that $Q[X]_{N(v)}=P[X]_{N(v)}$ by Lemma \ref{u}.
Therefore, $Q=P$ is homogeneous.
\end{proof}

\begin{corollary}\label{primehomo}
Let $R =\bigoplus_{\alpha \in \Gamma}R_{\alpha}$ be a graded
integral domain with property $(\#)$. If $R$ is a UMT-domain, then every prime ideal of
$R_{N(H)}$ is extended from a homogeneous $t$-ideal of $R$.
\end{corollary}

\begin{proof}  
Let $Q'$ be a nonzero prime ideal of $R_{N(H)}$. Then $Q'=QR_{N(H)}$ for some
prime ideal $Q$ of $R$. Note that $Q\subseteq M$ for some
homogeneous maximal $t$-ideal $M$ of $R$ because $R$ satisfies
property $(\#)$. Thus, $Q$ is a homogeneous prime $t$-ideal by Proposition \ref{homo}.
\end{proof}

\begin{corollary}
Let $R =\bigoplus_{\alpha \in \Gamma}R_{\alpha}$ be a UMT-domain,
$H$ be the set of nonzero homogeneous elements of $R$, and
$Q$ be a prime $t$-ideal of $R$ with $Q \cap H = \emptyset$.
Then $R_Q$ is a valuation domain. 
\end{corollary}

\begin{proof}
Without restriction let $Q \neq (0)$.
If $C(Q)_t \subsetneq R$, then $Q$ is homogeneous by Proposition \ref{homo}, and hence
$Q \cap H \neq \emptyset$, a contradiction. Thus, $C(Q)_t = R$,
and hence $R_Q$ is a valuation domain by Theorem \ref{theorem 1}.
\end{proof}

\begin{remark}
{\em The converse of Corollary \ref{primehomo} is not true in general.
For example, let $D$ be an integral domain which is not a UMT-domain
such that each maximal ideal of $D$ is a $t$-ideal (e.g.,
$D=\mathbb{Q}+Y\mathbb{R}[Y]$, where $\mathbb{Q}$ (resp. $\mathbb{R}$) is
the field of rational (resp. real) numbers and $Y$ is an
indeterminate over $\mathbb{R}$). Let $K$ be the quotient field of
$D$, $K[X]$ be the polynomial ring over $K$, and $R = D+XK[X]$,
i.e., $R =\{ f\in K[X]\mid f(0)\in D\}$. Then $R$ is a graded
integral domain with property ($\#$)  \cite[Corollary
9]{c16} and every ideal of $R_{N(H)}$ is extended from a
homogeneous ideal of $R$ \cite[Lemma 6 and Corollary 8]{c16}. Note
that $R$ is not a UMT-domain  \cite[Proposition 3.5]{fgh98}.}
\end{remark}

As in \cite{park}, we say that $R =\bigoplus_{\alpha \in
\Gamma}R_{\alpha}$ is a {\em graded strong Mori domain} (graded SM
domain) if $R$ satisfies the ascending chain condition on
homogeneous $w$-ideals. Clearly, $R$ is a graded SM domain
if and only if every homogeneous
$w$-ideal $A$ of $R$ has finite type, i.e., $A = I_w$ for some finitely generated ideal $I$ of $R$.
It is known that a graded SM domain
satisfies property $(\#)$ (see the proof of \cite[Theorem
2.7]{ac13}).

\begin{lemma} \label{krull}
Let $R =\bigoplus_{\alpha \in \Gamma}R_{\alpha}$ be a graded SM
domain and $R^*$ be the complete integral closure of $R$. 
Then $R^{[w]} = R^*$ and $R^{[w]}$ is a graded Krull domain. 
\end{lemma}

\begin{proof}
Let $\Omega$ be the set of homogeneous maximal $t$-ideals of $R$.
Note that if $Q \in \Omega$, then $R_{H \setminus Q}$ is a graded Noetherian domain \cite[Theorem 3.5]{park},
and hence $\overline{R_{H \setminus Q}} = (R_{H \setminus Q})^*$ \cite[Lemma 2.3]{park}.
Thus, $R^* = \bigcap_{Q \in \Omega}(R_{H \setminus Q})^*
= \bigcap_{Q \in \Omega}\overline{R_{H \setminus Q}} =
\bigcap_{Q \in \Omega}\bar{R}_{H \setminus Q} = R^{[w]}$,
where the first equality is from \cite[Proof of Corollary 3.6]{park}
and the last equality is by Theorem \ref{w-integral},
and hence $R^{[w]}$ is a graded Krull domain \cite[Corollary 3.6]{park}. 
\end{proof}

\begin{corollary} \label{sm domain}
Let $R =\bigoplus_{\alpha \in \Gamma}R_{\alpha}$ be a graded SM
domain. Then the following statements are equivalent.
\begin{enumerate}
\item $R$ is a UMT-domain.
\item $R_{N(H)}$ is one-dimensional.
\item Every overring of $R_{N(H)}$ is a Noetherian domain.
\item If $Q$ is a nonzero prime ideal of $R^{[w]}$ such that
$Q \cap R$ is homogeneous and $(Q \cap R)_t \subsetneq R$, then $Q_t \subsetneq R^{[w]}$.
\end{enumerate}
\end{corollary}

\begin{proof}
$(1)\Rightarrow(2)$ Let $Q$ be a homogeneous maximal $t$-ideal of
$R$. Let $P$ be a nonzero prime ideal of $R$ such that $P \subseteq Q$.
Then $P$ is a homogeneous prime $t$-ideal by Proposition \ref{homo},
and hence $P= J_w$ for some nonzero finitely generated ideal $J$ of $R$.
If $x \in P$, then $xI \subseteq J$  for some nonzero finitely generated ideal $I$ of $R$ with $I_t = R$,
and thus $x \in xR_Q = xI_Q \subseteq J_Q$ because $I \nsubseteq Q$. Therefore, $P_Q = J_Q$
is finitely generated. Thus, every prime ideal of $R_Q$ is finitely generated, and hence  
$R_Q$ is Noetherian. 
Also, since $R_Q$ is a quasi-Pr\"ufer domain,
$R_Q$ is one-dimensional \cite[Theorem 3.7]{hz89}. Note that
Max$(R_{N(H)}) = \{Q_{N(H)} \mid Q$ is a homogeneous maximal
$t$-ideal of $R\}$ and $(R_{N(H)})_{Q_{N(H)}} = R_Q$ for all
homogeneous maximal $t$-ideals $Q$ of $R$. Thus, $R_{N(H)}$ is one-dimensional.

$(2) \Rightarrow (3)$ If $R_{N(H)}$ is one-dimensional, then every
homogeneous maximal $t$-ideal of $R$ has height-one. Hence,
$R_{N(H)}$ is one-dimensional Noetherian \cite[Theorem 2.7]{ac13},
and thus every overring of $R_{N(H)}$ is Noetherian \cite[Theorem
93]{kap}.

$(3) \Rightarrow (1)$ By \cite[Exercise 20, p.64]{kap}, $R_{N(H)}$
is one-dimensional, and thus $R_{N(H)}$ is a UMT-domain
\cite[Theorem 3.7]{hz89}. Thus, $R$ is a UMT-domain by Corollary
\ref{sharp}.

(1) $\Leftrightarrow$ (4) This follows directly from Corollary \ref{coro 1} and Lemma \ref{krull}.
\end{proof}


\section{Monoid domains and UMT-domains}
 
Let $D$ be an integral domain with quotient field $K$, $\Gamma$ be a torsionless commutative cancellative monoid
with quotient group $\langle\Gamma \rangle$, $\bar{D}$ be the integral closure of $D$,
$\bar{\Gamma}$ be the integral closure of $\Gamma$, and
$D[\Gamma]$ be the monoid domain of $\Gamma$ over $D$. Clearly, $D[\Gamma]$
is a $\Gamma$-graded integral domain with deg$(aX^{\alpha}) = \alpha$ for all $0 \neq a \in D$ and
$\alpha \in \Gamma$. 
Note that $\overline{D[\Gamma]} = \bar{D}[\bar{\Gamma}]$ \cite[Theorem 12.10]{g84}.
Also, $D[\Gamma]$ satisfies property $(\#)$ \cite[Example 1.6]{ac13}.
As in Section 2, let $H$ be the set of nonzero homogeneous elements of $D[\Gamma]$
and $N(H) = \{f \in D[\Gamma] \mid C(f)_v = D[\Gamma]\}$.

\begin{theorem} \label{monoid domain}
The following statements are equivalent for $D[\Gamma]$.
\begin{enumerate}
\item $D[\Gamma]$ is a UMT-domain.
\item $D[\Gamma]_{N(H)}$ is a quasi-Pr\"ufer domain.
\item $\overline{D}[\overline{\Gamma}]_{N(H)}$ is a Pr\"ufer domain.
\item $D$ is a UMT-domain and $\bar{\Gamma}_S$ is a
valuation monoid for all maximal $t$-ideals $S$ of $\Gamma$.
\end{enumerate}
\end{theorem}

\begin{proof}
(1) $\Leftrightarrow$ (2) $\Leftrightarrow$ (3) This is an immediate
consequence of Corollary \ref{sharp}.

(1) $\Rightarrow$ (4) Let $P$ be a maximal $t$-ideal of $D$. Then
$P[\Gamma]$ is a homogeneous maximal $t$-ideal of $D[\Gamma]$ \cite[Corollary 1.3]{ac05}, and hence
$\overline{D[\Gamma]}_{P[\Gamma]}$ is a Pr\"ufer domain by
Corollary \ref{coro 1}. Note that
$\overline{D[\Gamma]}_{P[\Gamma]} =
\overline{D_P[\Gamma]}_{PD_P[\Gamma]} = \overline{D_P}[\langle
\Gamma \rangle]_{PD_P[\langle \Gamma \rangle]}$. Let $I$ be a nonzero finitely generated ideal of
$\overline{D_P}$. Then
\begin{eqnarray*}
\overline{D_P}[\langle\Gamma \rangle]_{PD_P[\langle \Gamma \rangle]}
&=& (I\overline{D_P}[\langle\Gamma \rangle]_{PD_P[\langle \Gamma
\rangle]})(I\overline{D_P}[\langle\Gamma \rangle]_{PD_P[\langle \Gamma \rangle]})^{-1}\\
&=& (I\overline{D_P}[\langle\Gamma \rangle]_{PD_P[\langle \Gamma \rangle]})((I\overline{D_P}[\langle\Gamma \rangle])^{-1}_{PD_P[\langle \Gamma \rangle]})\\
&=&  (I\overline{D_P}[\langle\Gamma \rangle]_{PD_P[\langle
\Gamma\rangle]})(I^{-1}\overline{D_P}[\langle\Gamma
\rangle]_{PD_P[\langle \Gamma \rangle]})\\
&=& (II^{-1})\overline{D_P}[\langle\Gamma
\rangle]_{PD_P[\langle \Gamma \rangle]},
\end{eqnarray*}
where the first equality holds because
$\overline{D[\Gamma]}_{P[\Gamma]}$ is a Pr\"{u}fer domain, the
second one holds by \cite[Lemma 3.4]{k89}, and the third one holds
by \cite[Lemma 2.3]{eik02}. Hence, $1 = \frac{g}{f}$ for some $f \in
D_P[\langle \Gamma \rangle] \setminus PD_P[\langle \Gamma \rangle]$
and $g \in (II^{-1})\overline{D_P}[\langle \Gamma \rangle]$.
Let $c_A(h)$ be the ideal of an integral domain $A$ generated by the coefficients of $h \in A[\langle \Gamma \rangle]$.
Note that $c_{D_P}(f) = D_P$, and so $c_{\overline{D_P}}(f) =
\overline{D_P}$. Hence,
$\overline{D_P}=c_{\overline{D_P}}(f)=c_{\overline{D_P}}(g)\subseteq II^{-1}\subseteq
\overline{D_P}$; so $II^{-1} = \overline{D_P}$. Thus, 
$\overline{D_P}$ is a Pr\"ufer domain.

Next, let $S$ be a maximal $t$-ideal of $\Gamma$. Then $D[S]$ is a homogeneous
maximal $t$-ideal of $D[\Gamma]$ \cite[Corollary 1.3]{ac05}, and
hence $\overline{D[\Gamma]}_{D[S]}$ is a Pr\"ufer domain by
Corollary \ref{coro 1}. Note that $D[\Gamma]_{D[S]} =
K[\Gamma]_{K[S]} = K[\Gamma_S]_{K[S_S]}$; so
$\overline{D[\Gamma]}_{D[S]} = K[\overline{\Gamma}_S]_{K[S_S]}$. Let
$\alpha, \beta \in \overline{\Gamma}_S$,
$J=(\alpha+\overline{\Gamma}_S)\cup(\beta+\overline{\Gamma}_S)$, 
$J^{-1}=\{\gamma-\delta\mid \gamma,
\delta\in\overline{\Gamma}_S$ and
$(\gamma-\delta)+J\subseteq\overline{\Gamma}_S\}$,
$R=K[\overline{\Gamma}_S]$, and $Q= K[S_S]$. Note that
$K[J]=(X^{\alpha},X^{\beta})$ is a finitely generated ideal of $R$. Then
$R_Q = (K[J]_Q)(K[J]_Q)^{-1} = (K[J]_Q)(K[J]^{-1})_Q =
(K[J]_Q)(K[J^{-1}]_Q) = K[JJ^{-1}]_Q$. Thus, $1 = \frac{g}{f}$ for
some $g \in K[JJ^{-1}]$ and $f \in K[\Gamma_S] \setminus Q$, and hence $f = g
\in K[JJ^{-1}]$. Note that $\overline{\Gamma}_S$ is the integral
closure of $\Gamma_S$; so $\Gamma_S = E_f \subseteq E_g \subseteq
JJ^{-1}$, where $E_f$ (resp., $E_g$) is the ideal of $\Gamma_S$ (resp., $\bar{\Gamma}_S$)
generated by the exponents of $f$ (resp., $g$).
Thus, $JJ^{-1} = \overline{\Gamma}_S$, which means that
$J$ is invertible. Hence, $J = \alpha + \overline{\Gamma}_S$ or $J =
\beta + \overline{\Gamma}_S$. Thus, $\overline{\Gamma}_S$ is a
valuation monoid.

(4) $\Rightarrow$ (1) Let $Q$ be a maximal $t$-ideal of $R =
D[\Gamma]$. We have to show that $R_Q$ is a quasi-Pr\"{u}fer domain.
If $Q \cap H = \emptyset$, then $R_Q$ is a valuation domain by
Theorem \ref{theorem 1}. Next, if $Q \cap H \neq \emptyset$, then either
$Q \cap D \neq (0)$ or $\{\alpha \in \Gamma \mid X^{\alpha} \in Q\} \neq \emptyset$.
Assume that $Q \cap D \neq (0)$, and let 
$Q \cap D = P$. Then $P$ is a maximal $t$-ideal of $D$ and $Q = P[\Gamma]$  \cite[Corollary 1.3]{ac05}. Hence, $\overline{D_P}$
is a Pr\"ufer domain by Corollary \ref{coro 1}, and hence if $S(H) =
\{f \in \overline{D_P}[\langle \Gamma \rangle] \mid C(f) =
\overline{D_P}[\langle \Gamma \rangle]\}$, then
$\overline{D_P}[\langle \Gamma \rangle]_{S(H)}$ is a Pr\"ufer domain
\cite[Example 4.2 and Corollaries 4.5 and 6.4]{acz16}. Note that
$S(H) = \overline{D_P}[\langle \Gamma \rangle] \setminus
\bigcup\{M[\langle \Gamma \rangle] \mid M$ is a maximal ideal of
$\overline{D_P}\}$; so Max$(\overline{D_P}[\langle \Gamma \rangle]_{S(H)})
= \{M[\langle \Gamma \rangle]_{S(H)} \mid M$ is a maximal ideal of $\overline{D_P}\}$, 
 and $\overline{D_P}[\langle \Gamma
\rangle]_{PD_P[\langle \Gamma \rangle]}$ is integral over
$D_P[\langle \Gamma \rangle]_{PD_P[\langle \Gamma \rangle]}$. Thus,
$\overline{D_P}[\langle \Gamma \rangle]_{PD_P[\langle \Gamma
\rangle]} = \overline{D_P}[\langle \Gamma \rangle]_{S(H)}$ is a
Pr\"ufer domain.

Finally, assume that $Q \cap D =(0)$, and let $S = \{\alpha \in \Gamma
\mid X^{\alpha} \in Q\}$. Then $S \neq \emptyset$, $S$ is a maximal
$t$-ideal of $\Gamma$, and $Q = D[S]$ \cite[Corollary 1.3]{ac05}.
We show that $\overline{D[\Gamma]}_{D[S]}$ is a valuation domain.
Indeed, note that $D[\Gamma]_{D[S]} = K[\Gamma_S]_{K[S_S]}$ and
$\overline{K[\Gamma_S]}_{K[S_S]} = K[\overline{\Gamma}_S]_{K[S_S]}$.
Clearly, if $J$ is the maximal ideal of $\overline{\Gamma}_S$, then
$K[J]$ is a unique prime ideal of $K[\overline{\Gamma}_S]$ lying
over $K[S_S]$, and hence $K[\overline{\Gamma_S}]_{K[S_S]} =
K[\overline{\Gamma_S}]_{K[J]}$. Also,
 $K[J]$ is a maximal $t$-ideal of $K[\overline{\Gamma_S}]$.
Note that every finitely generated homogeneous ideal of $K[\overline{\Gamma_S}]$
is principal because $\overline{\Gamma_S}$ is a valuation monoid; so $K[\overline{\Gamma_S}]$ is a
P$v$MD by Corollary \ref{pvmd}. Thus,
$\overline{D[\Gamma]}_{D[S]} = K[\overline{\Gamma_S}]_{K[J]}$ is a
valuation domain.
\end{proof}

\begin{corollary}
If $\bar{\Gamma}$ is a valuation monoid $($e.g., $\Gamma$ is a
numerical semigroup or a group$)$, then $D[\Gamma]$ is a UMT-domain if
and only if $D$ is a UMT-domain.
\end{corollary}

\begin{proof}
Let $S$ be a maximal $t$-ideal of $\Gamma$. Then $\bar{\Gamma}_S$ is a 
valuation monoid, and thus the result follows directly from Theorem \ref{monoid domain}.
\end{proof}

\begin{corollary}
{\em (\cite[Theorem 2.4]{fgh98} and \cite[Corollary 3.14]{cpre})} Let 
$D$ be an integral domain and $\{X_{\alpha}\}$ be a nonempty
set of indeterminates over $D$. Then the following statements are equivalent.
\begin{enumerate}
\item $D$ is a UMT-domain.
\item $D[\{X_{\alpha}\}]$ is a UMT-domain.
\item $D[\{X_{\alpha}, X_{\alpha}^{-1}\}]$ is a UMT-domain.
\end{enumerate}
\end{corollary}

\begin{proof}
Let $\mathbb{N}_{\alpha}$ be the additive monoid of nonnegative integers
and 
$\Gamma = \bigoplus_{\alpha}\mathbb{N_{\alpha}}$.
Then $\Gamma$ is a unique factorization monoid,
$D[\Gamma] = D[\{X_{\alpha}\}]$, and $D[\langle \Gamma \rangle] = D[\{X_{\alpha}, X_{\alpha}^{-1}\}]$.
Thus, the result is an immediate consequence of Theorem \ref{monoid domain}.
\end{proof}

We say that $\Gamma$ is a Pr\"ufer $v$-multiplication semigroup (P$v$MS)
if every finitely generated ideal of $\Gamma$ is $t$-invertible; equivalently, $\Gamma_S$ is a valuation monoid
for all maximal $t$-ideals $S$ of $\Gamma$ \cite[Theorem 17.2]{hk}. Thus, by Theorem \ref{monoid domain}, we have

\begin{corollary} \cite[Proposition 6.5]{AA}
$D[\Gamma]$ is a P$v$MD if and only if $D$ is a
P$v$MD and $\Gamma$ is a P$v$MS.
\end{corollary}

\vspace{.2cm}
\noindent {\bf Acknowledgement.}
The authors would like to thank the referee for his/her careful reading of the manuscript and several wonderful comments
which greately improved the paper.

\end{document}